\documentclass[11pt]{article}
\usepackage[english]{babel}
\usepackage[latin1]{inputenc}
\usepackage[T1]{fontenc}

\usepackage{graphics,graphicx} 

\usepackage{latexsym}
\usepackage{amsfonts}
\usepackage{amsthm}
\usepackage{amsmath}
\usepackage{graphicx}
\usepackage[hmargin=2.3cm,vmargin=2.5cm]{geometry}

\usepackage{tikz}
\usetikzlibrary{decorations.pathreplacing,arrows,automata}

\tikzstyle{vertex}=[circle,fill=black!100,text=white,inner sep=0.8mm]
\tikzstyle{point}=[circle,fill=black,inner sep=0.1mm]

\newtheorem{theorem}{Theorem}
\newtheorem{lemma}{Lemma}

\newtheorem{corollary}{Corollary}

\date{}

\title{New results on word-representable graphs}
\author{Andrew Collins\thanks{Mathematics Institute, University of Warwick, Coventry, CV4 7AL, UK. 
Email: A.Collins.2@warwick.ac.uk.} 
\and 
Sergey Kitaev\thanks{Department of Computer and Information Sciences, University of Strathclyde, Glasgow, G1 1XH, UK. Email: sergey.kitaev@cis.strath.ac.uk}
\and 
Vadim V. Lozin\thanks{Mathematics Institute and DIMAP, University of Warwick, Coventry, CV4 7AL, UK. 
Email: V.Lozin@warwick.ac.uk.}}

\begin{document}

\maketitle
\thispagestyle{empty}

\begin{abstract} A graph $G=(V,E)$ is word-representable if there exists a word
$w$ over the alphabet $V$ such that letters $x$ and $y$ alternate in $w$ if and 
only if $(x,y)\in E$ for each $x\neq y$. The set of word-representable graphs
generalizes several important and well-studied graph families, such as 
circle graphs, comparability graphs, 3-colorable graphs, graphs of vertex degree 
at most 3, etc. By answering an open question from \cite{HalKitPya2}, in the present 
paper we show that not all graphs of vertex degree at most 4 are word-representable.
Combining this result with some previously known facts, we derive that the number of 
$n$-vertex word-representable graphs is  $2^{\frac{n^2}{3}+o(n^2)}$.
\end{abstract}


\section{Introduction}
Let $G=(V,E)$ be a simple (i.e. undirected, without loops and multiple edges) graph
with vertex set $V$ and edge set $E$. We say $G$ is {\it word-representable} if there exists 
a word $w$ over the alphabet $V$ such that letters $x$ and $y$ alternate in $w$ if and 
only if $(x,y)\in E$ for each $x\neq y$. 

The notion of word-representable graphs has its roots in the study of the celebrated Perkins semigroup \cite{KitSei,S}.
These graphs possess many attractive properties (e.g. a maximum clique in such graphs can be found in polynomial time), 
and they provide a common generalization of several important graph families,  
such as circle graphs, comparability graphs, 3-colorable graphs, graphs of vertex degree 
at most 3 (see \cite{classes} for definitions of these families). 

Recently, a number of fundamental results on word-representable graphs were obtained 
in the literature \cite{HalKitPya1,HalKitPya2,KitPya,KitSalSevUlf1,KitSalSevUlf2}. In particular, 
Halld\'orsson et al. \cite{HalKitPya2} have shown that a graph is word-representable if and only 
if it admits a {\em semi-transitive orientation}. 
However, our knowledge on these graphs is still very limited and many important questions remain open. For example,
how hard is it to decide whether a given graph is 
word-representable or not? What is the minimum length of a word that represents a given graph?
How many word-representable graphs on $n$ vertices are there? Does this family include all graphs
of vertex degree at most 4? 

The last question was originally asked in \cite{HalKitPya2}. In the present paper we answer this question 
negatively by exhibiting a graph of vertex degree at most 4 which is not word-representable. 
This result allows us to obtain an upper bound on the asymptotic growth of the number of
$n$-vertex word-representable graphs. Combining this result with a lower bound that follows from some previously known
facts, we conclude that the number of $n$-vertex word-representable graphs is $2^{\frac{n^2}{3}+o(n^2)}$.

All preliminary information related to the notion of word-representable graphs can be found in Section~\ref{sec:pre}.
In Section~\ref{sec:deg4}, we prove our negative result about graphs of degree at most 4 and in Section~\ref{sec:main}, 
we derive the asymptotic formula on the number of word-representable graphs.


\section{Word-representable graphs: definition, examples and related results}
\label{sec:pre}

Distinct letters $x$ and $y$  {\em alternate} in a word $w$ if the deletion of all other letters
from the word results in either $xyxy\cdots$ or $yxyx\cdots$. 
A graph $G=(V,E)$ is \emph{word-representable} if there exists a word
$w$ over the alphabet $V$ such that letters $x$ and $y$ alternate in
$w$ if and only if $(x,y)\in E$ for each $x\neq y$. For example, 
the graph $M$ in Figure \ref{fig:representable-graphs} is word-representable, 
because the word $w=1213423$ has the right alternating properties, 
i.e. the only non-alternating pairs in this word are 1,3 and 1,4 that 
correspond to the only non-adjacent pairs of vertices in the graph. 

If a graph is word-representable, then there are infinitely many words representing it. 
For instance,  the complete graph $K_4$ in Figure \ref{fig:representable-graphs} can be represented by words
1234, 3142, 123412, 12341234, 432143214321, etc. In general, to represent a complete graph on $n$ vertices, 
one can start with writing up any permutation of length $n$ and adjoining from either side any number of copies 
of this permutation.  

\begin{figure}[ht]
\begin{center}
\includegraphics[scale=0.7]{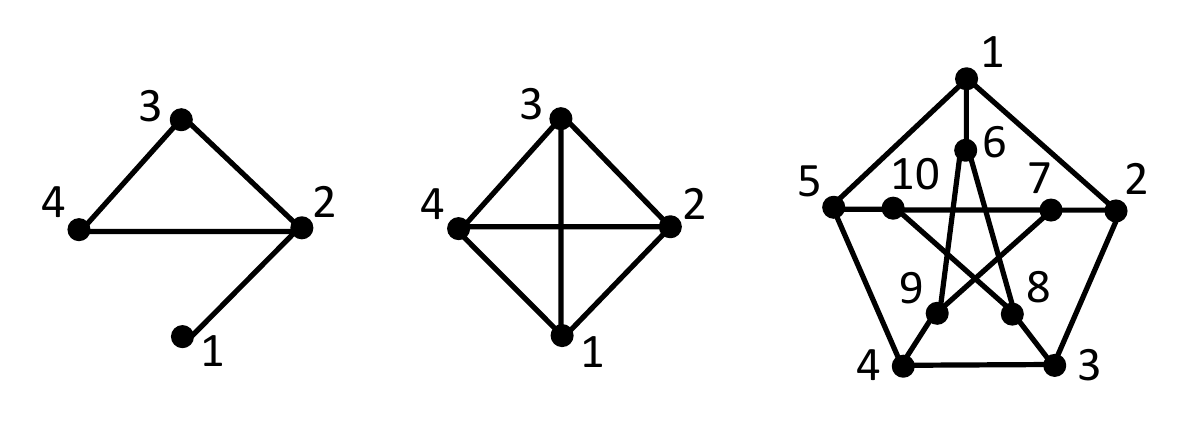}
\end{center}
\vspace{-20pt}
\caption{Three word-representable graphs $M$ (left), the complete graph $K_4$ (middle), and the Petersen graph (right)}
\label{fig:representable-graphs}
\end{figure}

If each letter appears exactly $k$ times in a word representing a graph, the graph is said to be \emph{$k$-word-representable}. 
It is known \cite{KitPya} that any word-representable graph is $k$-word-representable for some $k$. 
For example, a $3$-representation of the Petersen graph shown in Figure \ref{fig:representable-graphs} 
is $$1387296(10)7493541283(10)7685(10)194562.$$
It is not difficult to see that a graph is 1-representable if and only if it is complete.
Also, with a bit of work one can show that a graph is 2-representable if and only if it is a circle graph,
i.e. the intersection graph of chords in a circle. Thus, word-representable graphs generalize both complete graphs and circle graphs.
They also generalize two other important graph families, comparability graphs and 3-colorable graphs.
This can be shown through the notion of semi-transitive orientation. 

A directed graph (digraph) $G=(V,E)$ is \emph{semi-transitive} if it has no directed cycles
and for any directed path $v_1v_2\cdots v_k$ with $k\geq 4$ and $v_i\in V$, 
either $v_1v_k\not\in E$ or $v_iv_j\in E$ for all $1\le i<j\le k$. 
In the second case, when $v_1v_k\in E$, we say that $v_1v_k$ is a shortcut. 
The importance of this notion is due to the following result proved in \cite{HalKitPya2}.

\begin{theorem}[\cite{HalKitPya2}]\label{thm:semi-trans}  
A graph is word-representable if and only if it admits a semi-transitive orientation. 
\end{theorem}

From this theorem and the definition of semi-transitivity it follows
that all comparability (i.e. transitively orientable) graphs are word-representable.
Moreover, the theorem implies two more important corollaries. 

\begin{theorem}[\cite{HalKitPya2}]\label{thm:3col} 
All $3$-colorable graphs are word-representable.
\end{theorem}

\begin{proof} 
Partitioning a 3-colorable graph in three independent sets, say I, II and III, 
and orienting all edges in the graph so that they are oriented from I to II and III, and from II to III, 
we obtain a semi-transitive orientation.
\end{proof}

\begin{theorem}[\cite{HalKitPya2}]\label{thm:3deg} 
All graphs of vertex degree at most $3$ are word-representable.
\end{theorem}

\begin{proof} 
By Brooks' Theorem, every connected graph of vertex degree at most 3, except for the complete graph $K_4$, 
is 3-colorable, and hence word-representable by Theorem~\ref{thm:3col}. 
Also, as we observed earlier, all complete graphs are word-representable. Therefore, all connected graphs of degree 
at most 3 and hence all graphs of degree at most 3 are word-representable.
\end{proof}

Whether all graphs of degree at most 4 are word-representable is a natural question following from Theorem~\ref{thm:3deg},
which was originally asked in \cite{HalKitPya2}.  In the next section, we settle this question negatively.


\section{A non-representable graph of vertex degree at most 4}
\label{sec:deg4}


The main result of this section is that the graph $A$ represented in Figure~\ref{fig:A} is not word-representable.
To prove this, we will show that this graph does not admit a semi-transitive orientation. 
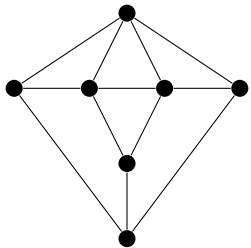
\begin{figure}[ht]
\begin{center}
	\begin{tikzpicture}
  		[scale=1,auto=left]

		\node[vertex] (n1) at (0,0)   { };
		\node[vertex] (n2) at (0,1)   { };
		\node[vertex] (n3) at (-1.5,2)   { };
		\node[vertex] (n4) at (-0.5,2)   { };
		\node[vertex] (n5) at (0.5,2)   { };
		\node[vertex] (n6) at (1.5,2)   { };
		\node[vertex] (n7) at (0,3)   { };

		\foreach \from/\to in {n1/n2, n1/n3, n1/n6, n2/n4, n2/n5, n7/n3, n7/n4, n7/n5, n7/n6, n3/n4, n4/n5, n5/n6}
    	\draw (\from) -- (\to);

	\end{tikzpicture}
\end{center}
\caption{The graph $A$}
\label{fig:A}
\end{figure}

\noindent
Our proof is a case analysis and the following lemma will be used frequently in the proof. 

\begin {lemma}\label{lem:1}
Let $D$ be a $K_4$-free graph admitting a semi-transitive orientation. 
Then no cycle of length $4$ in this orientation has three  consecutively oriented edges. 
\end{lemma}

\begin{proof}
If a semi-transitive orientation of $D$ contains a cycle of length four with three consecutively oriented edges,
then the fourth edge has to be oriented in the opposite direction to avoid an oriented cycle. 
However, the fourth edge now creates a shortcut. Hence the cycle must contain both chords to make it transitive, 
which is impossible because $D$ is $K_4$-free.
\end{proof}

\begin{theorem}\label{thm:A}
The graph $A$ does not admit a semi-transitive orientation.
\end{theorem}

\begin{proof} 
In order to prove that $A$ does not admit a semi-transitive orientation, we will explore {\it all} orientations 
of this graph and will show that each choice leads to a contradiction. At each step of the proof we choose 
a vertex and split the analysis into two cases depending on the orientation of the chosen edge. 
The chosen edge and its orientation will be shown by a solid arrow 
(\begin{tikzpicture} \draw [-latex] (0,0.1) -- (0.5,0.1); \draw [-] [white] (0,0) -- (0.01,0); \end{tikzpicture}). 
This choice of orientation may lead to other edges having an orientation assigned to them to satisfy Lemma~\ref{lem:1}. 
The orientations forced by the solid arrow through an application of  Lemma~\ref{lem:1} will be shown by means of double
arrows (\begin{tikzpicture} \draw [->>] (0,0.1) -- (0.5,0.1); \draw [-] [white] (0,0) -- (0.01,0); \end{tikzpicture}).
When we use Lemma~\ref{lem:1} to derive a double arrow, we always apply it with respect to a particular cycle of length 4. 
This cycle will be indicated by four white vertices. Since the graph $A$ has many cycles of length 4, repeated use 
of Lemma~\ref{lem:1} applied to different cycles may lead to a contradiction, where one more cycle of length 4 
has three consecutively oriented edges. We will show that in all possible cases a contradiction of this type arises. 
The proof is illustrated by diagrams.

\medskip
\noindent
{\it Case} 1: We start by choosing the orientation for the edge indicated in the diagram below.

\begin{center}


\end{center}

\medskip
\noindent
{\it Case} 1.1: Now we orient one more edge (solid arrow), which leads to two more orientations being assigned due to Lemma 1
(double arrows).

\begin{center}
	%

\end{center}

\medskip
\noindent
{\it Case} 1.1.1: One more edge is oriented (solid arrow) and this choice leads to a contradiction through repeated use of Lemma 1
(a cycle of four white vertices with three consecutively oriented edges in the final of the four diagrams below).

\begin{center}
	%

\end{center}

\medskip
\noindent
{\it Case} 1.1.2: In this case the orientation of the edge chosen in Case 1.1.1 is reversed.

\begin{center}
	%

\end{center}

\medskip
\noindent
{\it Case} 1.1.2.1: By orienting one more edge (solid arrow) we obtain a contradiction through repeated use of Lemma 1 
(a cycle of four white vertices with three consecutively oriented edges in the final diagram).

\begin{center}
	%

\end{center}

\medskip
\noindent
{\it Case} 1.1.2.2: Now we reverse the orientation of the edge chosen in Case 1.1.2.1. 

\begin{center}
	%

\end{center}

\medskip
\noindent
{\it Case} 1.1.2.2.1: By orienting one more edge (solid arrow) we obtain a contradiction through repeated use of Lemma 1.

\begin{center}
	%

\end{center}

\medskip
\noindent
{\it Case} 1.1.2.2.2: By reversing the orientation of the edge chosen in Case 1.1.2.2.1 we obtain a contradiction again. This completes 
the analysis of Case 1.1.

\begin{center}
	%

\end{center}

\medskip
\noindent
{\it Case} 1.2: The orientation of the edge chosen in Case 1.1 is reversed. 

\begin{center}
	%

\end{center}

\medskip
\noindent
{\it Case} 1.2.1: We orient one more edge (solid arrow) and derive one consequence (double arrow).

\begin{center}
	%

\end{center}

\medskip
\noindent
{\it Case} 1.2.1.1: One more edge is oriented (solid arrow) leading to a contradiction. 

\begin{center}
	%

\end{center}

\medskip
\noindent
{\it Case} 1.2.1.2: The orientation of the edge chosen in Case 1.2.1.1 is reversed. 

\begin{center}
	%

\end{center}

\medskip
\noindent
{\it Case} 1.2.1.2.1: One more edge is oriented (solid arrow) leading to a contradiction.  

\begin{center}
	%

\end{center}

\medskip
\noindent
{\it Case} 1.2.1.2.2: The orientation of the edge chosen in Case 1.2.1.2.1 is reversed leading a contradiction again. This 
completes the analysis of Case 1.2.1.

\begin{center}
	%

\end{center}

\medskip
\noindent
{\it Case} 1.2.2: The orientation of the edge chosen in Case 1.2.1 is reversed (solid arrow) and one consequence is derived (double arrow). 

\begin{center}
	%

\end{center}

\medskip
\noindent
{\it Case} 1.2.2.1: One more edge is oriented (solid arrow) leading to a contradiction. 

\begin{center}
	%

\end{center}

\medskip
\noindent
{\it Case} 1.2.2.2: The orientation of the edge chosen in Case 1.2.2.1 is reversed. 

\begin{center}
	%

\end{center}

\medskip
\noindent
{\it Case} 1.2.2.2.1: One more edge is oriented (solid arrow) leading to a contradiction.

\begin{center}
	%

\end{center}

\medskip
\noindent
{\it Case} 1.2.2.2.2: The orientation of the edge chosen in Case 1.2.2.2.1 is reversed.

\begin{center}
	%

\end{center}

\medskip
\noindent
{\it Case} 1.2.2.2.2.1: One more edge is oriented (solid arrow) leading to a contradiction.

\begin{center}
	%

\end{center}

\medskip
\noindent
{\it Case} 1.2.2.2.2.2: The orientation of the edge chosen in Case 1.2.2.2.2.1 is reversed, which leads to a contradiction again. 
This completes the analysis of Case 1.

\begin{center}
	%

\end{center}

\medskip
\noindent
{\it Case} 2. In this case, we reverse the orientation of the edge chosen in Case 1 and complete the proof by symmetry,
i.e. by reversing the orientations obtained in Case 1.
\end{proof}


\section{Asymptotic enumeration of word-representable graphs}
\label{sec:main}

Given a class $X$ of graphs, we write $X_n$ for the number of graphs in $X$ 
with vertex set $\{1,2,\ldots,n\}$, i.e. the number of {\it labelled} graphs in $X$. 
Following \cite{SpHerProp}, we call $X_n$ the {\it speed} of $X$. 

A class $X$ is {\it hereditary} if it is closed under taking induced subgraphs.
Alternatively, $X$ is hereditary if $G\in X$ implies $G-x\in X$ for any vertex 
$x\in V(G)$. Clearly, if $G$ is a word-representable graph and $w$ is a word 
representing $G$, then for any vertex $x\in V(G)$ the word obtained from $w$
by deleting all appearances of $x$ represents $G-x$. This observation leads
to the following obvious conclusion.

\begin{theorem}
The class of word-representable graphs is hereditary.
\end{theorem}

The speed of hereditary classes and their asymptotic structure have been extensively studied in the literature.
In particular, it is known that for every 
hereditary class $X$ different from the class of all finite graphs,
\begin{equation*}\label{eq:0}
\lim\limits_{n\to\infty}\frac{\log_2 X_n}{\binom{n}{2}}=1-\frac{1}{k(X)},
\end{equation*}
where $k(X)$ is a natural number, called the {\it index} of $X$. To define this notion 
let us denote by ${\cal E}_{i,j}$ the class of graphs whose vertices can be partitioned
into at most $i$ independent sets and $j$ cliques. In particular, ${\cal E}_{p,0}$
is the class of $p$-colorable graphs.
Then $k(X)$ is the largest $k$ such that $X$ contains ${\cal E}_{i,j}$ with $i+j=k$ for some $i$ and $j$. 
This result was obtained independently by Alekseev \cite{Ale92} and Bollob\'{a}s and Thomason  \cite{BT95,BT97}
and is known nowadays as the Alekseev-Bollob\'{a}s-Thomason Theorem (see e.g. \cite{ABT-theorem}).

Since $\binom{n}{2}$ is the minimum number of bits needed to represent an arbitrary 
$n$-vertex labeled graph and $\log_2 X_n$ is the minimum number of bits needed to represent an
$n$-vertex labeled graph in the class $X$,
the ratio $\log_2 X_n/{\binom{n}{2}}$ can be viewed as the coefficient of compressibility
for representing graphs in $X$ and its asymptotic value was called by Alekseev \cite{Ale82} 
the {\it entropy} of $X$.   

\medskip
We now apply the Alekseev-Bollob\'{a}s-Thomason Theorem in order to derive an asymptotic formula
for the number of word-representable graphs. We start with the number of $n$-vertex {\it labelled} 
graphs in this class, which we denote by $b_n$.

\begin{theorem}\label{thm:bn} 
$$\lim_{n\to\infty}\frac{\log_2 b_n}{{n\choose 2}}=\frac{2}{3}.$$
\end{theorem}

\begin{proof}
By Theorem~\ref{thm:3col}, ${\cal E}_{3,0}$ is a subclass of the class of word-representable graphs
and hence its index is at least 3. In order to show that the index does not exceed 3, we observe that 
the graph $A$ represented in Figure~\ref{fig:A} belongs to all minimal classes of index 4,
and hence the family of word-representable graphs does not contain any of these minimal classes by 
Theorems~\ref{thm:semi-trans} and~\ref{thm:A}.
Therefore, the index of the class of word-representable graphs is precisely 3. 
\end{proof}

\medskip
We now proceed to the number of unlabelled  $n$-vertex word-representable graphs, which we denote by $a_n$.

\begin{theorem}\label{thm:an} 
$$\lim_{n\to\infty}\frac{\log_2 a_n}{{n\choose 2}}=\frac{2}{3}.$$
\end{theorem}
\begin{proof}
Clearly, $b_n\le n!a_n$ and $\log_2 n!\le\log_2 n^n =n\log_2 n$.
Therefore, 
$$\lim_{n\to\infty}\frac{\log_2 b_n}{{n\choose 2}}\le \lim_{n\to\infty}\frac{\log_2 (n!a_n)}{{n\choose 2}}=
\lim_{n\to\infty}\frac{\log_2 n! + \log_2 a_n}{{n\choose 2}}\le
\lim_{n\to\infty}\frac{n\log_2 n + \log_2 a_n}{{n\choose 2}}=\lim_{n\to\infty}\frac{\log_2 a_n}{{n\choose 2}}.$$

On the other hand, obviously $b_n\ge a_n$ and hence $\lim_{n\to\infty}\frac{\log_2 b_n}{{n\choose 2}}\ge\lim_{n\to\infty}\frac{\log_2 a_n}{{n\choose 2}}$.
Combining, we obtain $\lim_{n\to\infty}\frac{\log_2 b_n}{{n\choose 2}}=\lim_{n\to\infty}\frac{\log_2 a_n}{{n\choose 2}}$.
Together with Theorem~\ref{thm:bn}, this proves the result.
\end{proof}

\begin{corollary}\label{cor-main}
$$a_n=2^{\frac{n^2}{3}+o(n^2)}.$$
\end{corollary}



\section{Acknowledgment}

The authors are grateful to the Edinburgh Mathematical Society for supporting 
the last authors' visit of the University of Strathclyde, during which parts of the presented results were obtained.
The last author also acknowledges support from DIMAP - the Center
for Discrete Mathematics and its Applications at the University of Warwick,
and from EPSRC, grant EP/I01795X/1.

\end{document}